\documentclass[11pt, oneside]{amsart}

\usepackage{tabularx, hyperref}
\usepackage{amssymb} \usepackage{amsfonts} \usepackage{amsmath}
\usepackage{amsthm} \usepackage{epsfig, subfig}
\usepackage{ amscd, amsxtra, latexsym}
\usepackage[all]{xy}
\usepackage{caption}
\usepackage{enumerate}
\usepackage{color}

\newtheorem{lemma}{Lemma}[section]
\newtheorem{thm}[lemma]{Theorem}
\newtheorem{prop}[lemma]{Proposition}

\theoremstyle{definition}
\newtheorem{defn}[lemma]{Definition}

\newtheorem{rem}[lemma]{Remark}
\newtheorem{conv}[lemma]{Convention}
\theoremstyle{definition}
\newtheorem*{claim}{Claim}

\definecolor{darkgreen}{cmyk}{1,0,1,.2}

\newcommand{\g} {\ensuremath {\gamma}}

\newcommand{\N}{\ensuremath {\mathbb{N}}}
\newcommand{\R} {\ensuremath {\mathbb{R}}}

\newcommand{\matH} {\ensuremath {\mathbb{H}}}

\newcommand{\calP} {\ensuremath {\mathcal{P}}}

\newcommand{\calF} {\ensuremath {\mathcal{F}}}
\newcommand{\calT} {\ensuremath {\mathcal{T}}}

\newcommand{\tilM}{\ensuremath{{\widetilde{M}}}}

\address{Mathematical Institute, 24-29 St Giles, Oxford OX1 3LB, United Kingdom}
\email{sisto@maths.ox.ac.uk}

\begin{document}

\title{$3-$manifold groups have unique asymptotic cones}
\author{Alessandro Sisto}
\maketitle

\begin{abstract}
We describe the (minimal) tree-graded structure of asymptotic cones of non-geometric graph manifold groups, and as a consequence we show that all said asymptotic cones are bilipschitz equivalent.
\par
Combining this with geometrization and other known results we obtain that all asymptotic cones of a given $3-$manifold group are bilipschitz equivalent.
\end{abstract}

\section{Introduction}
Asymptotic cones of groups are useful asymptotic invariants defined in \cite{vDW} generalizing a construction by Gromov \cite{G1}. They depend on the choice of an ultrafilter and a diverging sequence of positive real numbers (the \emph{scaling factors}). Examples of all sorts of exotic behaviors with respect to changing the ultrafilter and/or scaling factor have been constructed within the class of finitely generated \cite{TV,OOS,DS,SS} and even that of finitely presented groups \cite{KSTT,OlS,OOH}. For example, it is proven in \cite{DS} that there exists a group with uncountably many non-homeomorphic asymptotic cones, while \cite{KSTT} contains examples of finitely presented groups all whose asymptotic cones are homeomorphic if the Continuum Hypothesis holds but that have $2^{2^{\aleph_0}}$ pairwise non-homeomorphic asymptotic cones if it fails. However, it is often the case that within a certain class of ``well-behaved'' groups the choice of ultrafilter/scaling factor does not matter. One of the simplest examples of this is provided by abelian groups. A more sophisticated example, which will be used later, is provided by hyperbolic groups \cite{DP}. Knowing something about the topology of the asymptotic cones of a group or class of groups can be used for several purposes, for example to obtain quasi-isometric rigidity results (e.g. for cocompact lattices in higher rank semisimple groups~\cite{KlL}, for fundamental groups of Haken manifolds \cite{KaL1,KaL2} and higher dimensional analogues \cite{FLS} and for mapping class groups~\cite{BKMM}). This is why it is interesting to know that the topology of the asymptotic cones of a certain group does not depend on the choice of ultrafilter/scaling factor.
\par
The aim of this paper is to show that this is the case for $3-$manifold groups. The following subsection contains the proof of our main results (Theorems \ref{graphcones} and \ref{3mancones}), up to some work that will be carried out later. The reader not familiar with the concepts used is referred to Section 2 for some background material.

\subsection{Main results}

The following results will be crucial:
\begin{enumerate}
\item If $G$ is hyperbolic relative to (proper) subgroups $H_1,\dots, H_n$ whose asymptotic cones are all bilipschitz equivalent, then all asymptotic cones of $G$ are bilipschitz equivalent (\cite[Theorem 6.3]{Si} or \cite{OS}).
\item If all asymptotic cones of $G$ contain cut-points, and the pieces in the minimal tree-graded structures are bilipschitz equivalent in each asymptotic cone, then all the asymptotic cones of $G$ are bilipschitz equivalent (\cite[Theorem 0.6]{Si}).
\end{enumerate}

Here is the first main result.

\begin{thm}\label{graphcones}
 All asymptotic cones of non-geometric graph manifolds are bilipschitz equivalent.
\end{thm}

\begin{proof}
 By $(2)$, we only need to show that the same property holds for pieces in the minimal tree-graded structures. This is done in Propositions \ref{pieces1} and \ref{pieces2}.
\end{proof}

\begin{thm}\label{3mancones}
 Let $M$ be a compact connected orientable $3-$manifold whose (possibly empty) boundary is a union of tori. Then the asymptotic cones of $\pi_1(M)$ are all bilipschitz equivalent.
\end{thm}

\begin{proof}
We will use the geometrization theorem \cite{Pe1,Pe2,KLo,MT,CZ}.
First of all, we can reduce to the case when $M$ is prime (i.e., it cannot be written as $N_1\# N_2$, where $N_i\neq S^3$). In fact, recall that $G_1*...*G_n$ is hyperbolic relative to $\{G_1,\dots,G_n\}$. So, if $M=M_1\#...\# M_n$ is the prime decomposition of $M$, and $\pi_1(M_i)$ has bilipschitz equivalent asymptotic cones for each $i$, then by $(1)$ so does $\pi_1(M)$.
\par
So, let us assume that $M$ is prime. Suppose first that $M$ is geometric. We list below the possible geometries, and a reference for the uniqueness of the asymptotic cones of the corresponding manifolds in case it is needed.
\begin{itemize}
 \item $S^3$.
 \item $\R^3$.
 \item $\matH^3$, empty boundary: \cite{DP}.
 \item $\matH^3$, non-empty boundary: $(1)$.
 \item $S^2\times \R$.
 \item $\matH^2\times \R$, both empty and non-empty boundary: \cite{DP} (and asymptotic cones preserving products).
 \item $\widetilde {SL_2\R}$: it is quasi-isometric to $\matH^2\times\R$.
 \item $Nil$: \cite{Pa}.
 \item $Sol$: \cite{dC}.
\end{itemize}

If $M$ is not geometric, then we have 2 cases:

\begin{itemize}
 \item $M$ is a graph manifold. This case has been dealt with in Theorem \ref{graphcones}.
 \item $M$ contains a hyperbolic component $N$. In this case $\pi_1(M)$ is hyperbolic relative to abelian and graph manifold groups (by the combination theorem in \cite{Da}). So we can apply what we already know combined with $(1)$.
\end{itemize}
\end{proof}

\subsection*{Acknowledgment} The author would like to thank Jason Behrstock and Cornelia Dru\c{t}u for helpful discussions.

\section{Background}

For more details on asymptotic cones see \cite{D}.

\begin{defn}
 A \emph{(non-principal) ultrafilter} on $\N$ is a finitely additive probability measure on $\calP(\N)$ with values in $\{0,1\}$ such that finite sets have measure $0$.
\end{defn}
We will only deal with non-principal ultrafilters, hence we shall omit the adjective non-principal.

\begin{defn}
 Let $(r_n)_{n\in\N}$ be a sequence of nonnegative real numbers and $\omega$ an ultrafilter on $\N$. The ultralimit $\omega-\lim r_n$ of the sequence $(r_n)$ is $r\in[0,+\infty]$ if for each neighborhood $U$ of $r$ we have $\omega(\{n:r_n\in U\})=1$.
\end{defn}
Ultralimits of sequences of real numbers always exist and are unique.
\begin{defn}
 Let $X$ be a metric space, $\omega$ an ultrafilter on $\N$, $r=(r_n)$ a sequence of positive real numbers such that $\omega-\lim r_n=+\infty$ and $x=(x_n)_{n\in\N}$ a sequence of points of $X$. Set
$$X^\omega_{r,x}=\left\{(y_n)\in X^\N:\omega-\lim \frac{d(x_n,y_n)}{r_n}<\infty\right\}.$$
The \emph{asymptotic cone}  $Cone_\omega\big(X,x,r\big)$ of $X$ with respect to ultrafilter $\omega$, the \emph{scaling factor} $(r_n)$ and the \emph{basepoint} $(x_n)$ is $X^\omega_{r,x}/_\sim$, where the equivalence relation $\sim$ is defined as
$$(y_n)\sim (z_n)\iff \omega-\lim d(y_n,z_n)/r_n= 0.$$
The metric on said asymptotic cone is defined as $d([(y_n)],[(z_n)])=\omega-\lim d(y_n,z_n)/r_n.$
\end{defn}

When referring to asymptotic cones of a (finitely generated) group we will always mean asymptotic cones of a Cayley graph of the group. Asymptotic cones of groups do not depend up to isometry on the choice of the basepoint and up to bilipschitz equivalence on the choice of a finite system of generators.
\par
We will need some definitions and results from \cite{DS}.

\begin{defn}
Let $X$ be a complete geodesic metric space and let $\mathcal{P}$ be a collection of closed geodesic subsets, called \emph{pieces}, which cover the space $X$. We say that $X$ is \emph{tree-graded} with respect to $\mathcal{P}$ if
\begin{itemize}
\item[(T1)] The intersection of any two different pieces is either empty or a single point.
\item[(T2)] Every simple geodesic triangle in $X$ is contained in one piece.
\end{itemize}
\end{defn}

\begin{defn}
\label{cutp}
Let $X$ be a geodesic metric space. A point $x \in X$ is called a \emph{cut-point} of $X$ if the space $X \backslash \{x\}$ is not path connected.
\end{defn}

\begin{lemma}[\cite{DS}, Lemma 2.15]
\label{cutpoints}
If $X$ is tree-graded and $A\subseteq X$ does not contain cut-points then $A$ is contained in a piece.
\end{lemma}

\begin{lemma}[\cite{DS}, Lemma 2.31]
Let $X$ be a complete geodesic metric space.
There exists a unique collection of subsets $\mathcal{P}$ of $X$, called the \emph{minimal tree-graded structure}, such that $X$ is tree-graded with respect to $\mathcal{P}$ and any $P\in\mathcal{P}$ is (either a singleton or) a set with no cut-points.
\end{lemma}

\begin{defn}
Let $X$ be a metric space. Fix an ultrafilter $\omega$ on $\N$, a scaling factor $r=(r_n)$ and a basepoint $x=(x_n)$. Let $\mathcal{A}$ be a collection of subsets of $X$. Then for every sequence $(A_n)$ of sets in $\mathcal{A}$, let
\[ \omega-\lim A_n = \{[(y_n)] \in Cone_\omega(X,x,r) : y_n \in A_n \}\]
be the \emph{ultralimit} of the sequence $(A_n)$.
We say that $X$ is \emph{asymptotically tree-graded} with respect to $\mathcal{A}$ if each $Cone_\omega(X,x,r)$ is tree-graded with respect to the set of non-empty pieces of the form
\[ \{ \omega-\lim A_n : (A_n)_{n \in \N} \in \mathcal{A}^\N \},\]
and also $\omega-\lim A_n\neq\omega-\lim A'_n$ if they are both non-empty and $\omega(\{n:A_n\neq A'_n\})=1$.
\end{defn}

The following is one of the several possible definitions of relative hyperbolicity.

\begin{defn}
 The (finitely generated) group $G$ is \emph{hyperbolic relative} to its subgroups $H_1,\dots,H_n$ if it is asymptotically tree-graded with respect to $\{gH_i\}_{g\in G,i=1,\dots,n}$.
\end{defn}

Notable examples of relatively hyperbolic groups, used in the proof of Theorem \ref{3mancones}, include free products, which are hyperbolic relative to the factors, and fundamental groups of finite volume hyperbolic manifolds (e.g. surfaces with punctures of negative Euler characteristic) which are hyperbolic relative to the cusp subgroups \cite{Fa}.

\section{Description of the pieces}

We start with a digression on the geometry of graph manifolds.

\subsection{Special paths in graph manifolds}

A preliminary definition:
\begin{defn}
 A tree of spaces is a pair $(X,T)$ where $X$ is a metric space, $T$ is a simplicial tree and for each vertex $v$ of $T$ a certain subset $X_v\subseteq X$, called \emph{vertex space}, has been assigned. For any edge $e$ of $T$ with endpoints $v,v'$ we will denote $X_e=X_{v}\cap X_{v'}$.
\end{defn}

Kapovich and Leeb defined in \cite{KL} a special class of graph manifolds, called \emph{flip}, which have the nice properties reported below. It will not be restrictive for us to study flip graph manifolds in view of the fact that any graph manifold group is quasi-isometric to the fundamental group of a flip graph manifold.
\par
Let $M$ be a flip graph manifold. There exists a locally $CAT(0)$ metric on $M$ such that its universal cover $\tilM$ has the structure of a tree of spaces $(\tilM,T)$ such that
\begin{itemize}
 \item for each vertex $v$, we have that $X_v$ is convex in $\tilM$ and isometric to $Y_v\times \R$ for some universal cover $Y_v$ of a compact surface with boundary,
 \item for each edge $e$, we have that $X_e$ is convex and isometric to $\R^2$,
 \item we can choose identifications of each $X_v$ with $Y_v\times \R$ and each $X_e$ with $\R^2$ such that, denoting the endpoints of the edge $e$ by $e^-, e^+$, for each $p\in Y_{e^-}$ (resp. geodesic $\gamma\in Y_{e^-}$) such that $(p,0)\in X_e$ (resp. $\gamma\times\{0\}\subseteq X_e$) we have that $\{p\}\times \R$ (resp. $\gamma\times\{t\}$ for each real $t$) is identified with $\gamma\times\{t\}\subseteq X_{e^+}$  for some geodesic $\gamma$ in $Y_{e^+}$ and $t\in\R$ (resp. $\{p\}\times\R\subseteq X_{e^+}$ for some $p\in Y_{e^+}$).
\end{itemize}

\begin{defn}
 We will say that a path in $\tilM$ is a \emph{special path} if it is constructed as follows. Let $x_0,x_n\in\tilM$ with $x_0\in X_{v_0},y_n\in X_{v_n}$. Let $v_0,\dots,v_n$ be the vertices on the unique geodesic in $T$ from $v_0$ to $v_n$. If $n=0$ let the special path connecting $x_0$ to $x_n$ be just a geodesic. Otherwise, let $e_i$ be the edge connecting $v_i$ to $v_{i+1}$. For $i=1,\dots n-1$ let $p_{i},q_i\in Y_{v_i}$ be the starting and final points of the unique geodesic from $\pi_{Y_{v_i}}(X_{e_{i-1}})$ to $\pi_{Y_{v_i}}(X_{e_{i}})$ minimizing the distance between those sets. Let $p_0=\pi_{Y_{v_0}}(x_0)$ and $q_0$ be the point in $\pi_{Y_{v_0}}(X_{e_{0}})$ minimizing the distance from $p_0$. Define $q_n$ similarly to $p_0$ and $p_n$ similarly to $q_0$. For $i=0,\dots,n-2$ let $t_{i+1},u_{i}\in\R$ be such that $(p_{i+1},t_{i+1})\in X_{v_{i+1}}$ is identified with $(q_{i},u_{i})\in X_{v_i}$. Let $t_0$ be such that $x_0=(p_0,t_0)$ and $u_n$ be such that $x_n=(q_n,u_n)$. For $i=0,\dots,n$ let $\g_i$ be the geodesic connecting $y_i=(p_i,t_i)$ to $z_i=(q_{i},u_{i})$.
\par
Finally, let the special path connecting $x_0$ to $x_n$ be the concatenation of the $\gamma_i$'s.

\end{defn}

The nice feature of special paths is the following:
\begin{rem}\label{nice}
If $n\geq 4$ then $\gamma_2,\dots,\gamma_{n-2}$ only depend on $v_0,v_n$. What is more, if for $i=1,2$ $\alpha_i$ is a special path connecting some point in $X_{w_i}$ to some point in $X_{w'_i}$ and the vertex $v$
\begin{enumerate}
 \item lies on the geodesic connecting $w_i,w'_i$, and
 \item $d(v,w_i),d(v,w'_i)\geq 2$,
\end{enumerate}
  then $\alpha_1\cap X_v=\alpha_2\cap X_v$.
\end{rem}

\begin{lemma}\label{special}
 There exists $K=K(M)$ such that all special paths are $K-$bilipschitz.
\end{lemma}

\begin{proof}
We will use the notation $a\approx b$ if there exists $K=K(M)$ such that $b/K\leq a\leq Kb$. In the notation of the previous definition, which we will use throughout, if $x,y\in X_{v_i}$ we will denote by $d^i_h$ the distance of their projections on $Y_{v_i}$ and by $d^i_v$ the distance of their projections on $\R$. Notice that $d(x,y)\approx d^i_h(x,y)+d^i_v(x,y)$.
Let $\delta$ be a geodesic connecting $x_0$ to $x_n$ (we can assume $n\geq 3$) and let $\g$ be the special path connecting them. We have
$$l(\g)\approx \sum_{i=0}^n (d^i_h(y_i,z_i)+d^i_v(y_i,z_i)),\ \ \ (1)$$
where we set $y_0=x_0$, $z_n=u_n$ for convenience.
Let $y'_i$ (resp. $z'_i$) be a point $K-$close to $\delta$ lying on $\{p_i\}\times\R$ (resp. $\{q_i\}\times \R$). Choose $y'_0=y_0$, $z'_n=z_n$. It is easily seen that
$$l(\delta)\approx \sum_{i=0}^n(d^i_h(y'_i,z'_i)+d^i_v(y'_i,z'_i))+\sum_{i=0}^{n-1} (d^i_h(z'_i,y'_{i+1})+d^i_v(z'_i,y'_{i+1})).\ \ \ (2)$$
Notice that if $x,y\in X_{e_i}$ then $d^i_h(x,y)=d^{i+1}_v(x,y)$ and similarly for $d^i_v$.
\par
Also, $d^i_h(y_i,z_i)=d^i_h(y'_i,z'_i)$, so we just have to analyze the other terms of the sums. Set $z_{-1}=z'_{-1}=y_0$ and $y_{n+1}=y'_{n+1}=z_n$. Notice that $d^i_v(y_i,z'_{i-1})=0$ as, for $i\geq 1$, $y_i=z_{i-1}$ and $z_{i-1},z'_{i-1}$ both belong to $\{q_{i-1}\}\times\R$. Similarly, $d^i_v(y'_{i+1},z_{i})=0$. Hence,
$$d^i_v(y_i,z_i)\leq d^i_v(y_i,z'_{i-1})+d^i_v(z'_{i-1},y'_{i+1})+d_v(y'_{i+1},z_i)=d^i_v(z'_{i-1},y'_{i+1})\leq$$
$$d^i_v(y'_i,z'_i)+d^i_v(z'_{i-1},y'_i)+d^i_v(z'_i,y'_{i+1}).\ \ \ (*)$$
Summing all inequalities $(*)$ and using $(1),(2)$ we get
$l(\g)\leq K l(\delta),$
and we are done (we should prove a similar inequality also for subpaths of $\g$, but those are special paths as well).

\end{proof}

\subsection{Clusters}

\begin{conv}
 We will denote by $Z$ the homogeneous real tree with valency $2^{\aleph_0}$ at each point (for the existence and uniqueness of such object see \cite{MNO,DP}).
\end{conv}

\begin{defn}
Let $T$ be a simplicial tree and for each vertex $v$ of $T$ let $Z_v$ denote a copy of $Z$, let $Q_v$ be $Z_v\times \R$ and let $\calF_v=\{\gamma_{v,e}\}_{e\in edge(T),v\in e}$ be a collection of bi-infinite geodesics in $Z_v$ indexed by the edges of $T$ containing $v$. The \emph{cluster (of copies of $Z\times \R$)} $X=X(T,\{\calF_v\})$ with data $(T,\{\calF_v\})$ is the metric space constructed as follows. Let $\sim$ be the equivalence relation on $\bigcup Q_v$ such that $x\sim y$ if either $x=y$ or, for some vertices $v,v'$, the edge $e$ connecting them and $u,v\in\R$, we have $x=(\gamma_{v,e}(t),u)$ and $y=(\gamma_{v',e}(u),t)$.
\par
Let $\hat{X}=\hat{X}(T,\{\calF_v\})$ be $\bigcup Q_v/_\sim$ endowed with the natural path metric.
\par
Finally, let $X$ be the metric completion of $\hat{X}$.
\end{defn}

\begin{defn}
 Let $F_k$ be a free group, and let $H_1,\dots,H_n$ be cyclic subgroups such that $F_k$ is hyperbolic relative to $H_1,\dots,H_n$. A \emph{marked tree} is (a tree isometric to) an asymptotic cone of $F_k$ endowed with the standard tree-graded structure.
\end{defn}

\begin{lemma}\label{theta}
 All marked trees are isometric through isometries that preserve the pieces. Furthermore, given marked trees $T_0$, $T_1$, pieces $P_0\subseteq T_0$, $P_1\subseteq T_1$ and an isometry $\theta:P_0\to P_1$, we can require the isometry $T_0\to T_1$ to extend $\theta$.
\end{lemma}

\begin{proof}
 The existence of the said isometry satisfying the required properties follows from \cite{Si} (or \cite{OS}), up to using a slightly improved version of \cite[Theorem 6.34]{Si} (uniqueness of universal tree-graded spaces), which in any case follows from the proof. We are referring to the proof rather than the statement because in our setting
\begin{itemize}
 \item the statement would not guarantee that there is an isometry preserving the pieces,
 \item the statement is not about extensions of ``partial'' isometric embeddings.
\end{itemize}
However, the isometries constructed in the proof do preserve pieces, and the proof strategy is indeed to extend partial isometries (the reader can readily check that the isometry we deal with is indeed one of those that are proved to be extendable).
\end{proof}

The following proposition probably admits a slightly simpler proof along the lines of \cite{BC}. However, special paths will be used in a forthcoming paper for other purposes.

\begin{prop}\label{pieces1}
 Each piece in the asymptotic cone of a non-geometric graph manifold (endowed with the minimal tree-graded structure) is bilipschitz homeomorphic to a cluster $X(T,\{\calF_v\})$ such that
\par\smallskip
 $(\ast)$ for each vertex $v$, $(Z_v,\{\gamma_{v,e}\})$ is a marked tree.
\end{prop}

\begin{proof}
It suffices to show the statement for the universal cover $\tilM$ of the flip graph manifold $M$. For $x,y\in \tilM$ denote by $d_{BS}(x,y)$ the distance in the Bass-Serre tree of the closest vertices $v,w$ such that $x\in X_v, y\in X_w$.
\par
 We have to prove is that pieces as in the statement satisfy a description similar to that of pieces in asymptotic cones of the mapping class group \cite{BKMM} and right angled Artin groups \cite{BC}, namely:
 
\begin{prop}
 $x,y\in \tilM_\omega$ belong to the same piece if and only if there exist points $x'=[(x'_n)],y'=[(y'_n)]$ arbitrarily close to $x,y$ such that $\omega-\lim d_{BS}(x'_n,y'_n)<+\infty$.
\end{prop}
\begin{proof}
The ``if'' part is easy. In fact, any subset of $\tilM_\omega$ of the form $\omega-\lim X_{v_n}$ is a subset of a piece, as it does not contain cut-points, see Lemma \ref{cutpoints}. Also, if $d(v_n,v'_n)=1$ for each $n$ then $\omega-\lim X_{v_n}$ and $\omega-\lim X_{v'_n}$ are subsets of the same piece. Finally, pieces are closed.
\par
Let us prove the ``only if'' part. We have to find a point $p\in\tilM_\omega$ with $p\neq x,y$ such that all paths from $x=[(x_n)]$ to $y=[(y_n)]$ pass through $p$. Define a path in $\tilM$ to be $\omega-$special if it is an ultralimit of special paths. Notice that Lemma \ref{special} tells us that one can approximate any continuous path in $\tilM_\omega$ by a concatenation of $\omega-$special paths. For $p\in\tilM$ denote by $v(p)$ some vertex such that $p\in X_{v(p)}$. We claim that we can find vertices $v_n, w_n$ such that
\begin{enumerate}
 \item $v_n, w_n$ lie on the geodesic from $v(x_n)$ to $v(y_n)$,
 \item $\omega-\lim d(v_n,w_n)=\infty$,
 \item there exist points $p_n\in X_{v_n}$ and $q_n\in X_{w_n}$ such that $d(p,q)=0$ where $p=[(p_n)],q=[(q_n)])$,
 \item $d(p,x), d(p,y)>0$.
\end{enumerate}
Let $(r_n)$ be the scaling factor of $\tilM_\omega$ and let $\epsilon$ be small enough that $x',y'$ as in the statement of the proposition do not exist in $B_\epsilon(x), B_\epsilon(y)$. Let $\g_n$ be a special path from $x_n$ to $y_n$, and let $p'_n,q'_n\in\g_n$ be such that $d(x_n,p'_n), d(y_n,q'_n)=\epsilon r_n$ (they can be defined $\omega-$a.e.). Subdivide the geodesic between $v(p'_n)$ and $v(q'_n)$ in (approximately) $\sqrt{d_n}$ disjoint intervals of length (approximately) $\sqrt{d_n}$, where $d_n=d(v(p'_n),v(q'_n))$. It is quite clear that we can choose $v_n,w_n$ to be the endpoints of one such interval because of the finiteness of $d(x,y)$.
\par
We are almost done. Consider any path $\alpha$ in $\tilM_\omega$ connecting $x$ to $y$. Fix some $\epsilon$ and consider a concatenation of $\omega-$special paths $\delta_1,\dots,\delta_k$ connecting $x$ to $y$ and contained in the $\epsilon-$neighborhood of $\alpha$. Write $\delta_i=\omega-\lim \delta^i_n$ where $\delta^i_n$ is a special path connecting $x^i_n$ to $y^i_n$. It is quite clear that for some $i$ we have $\omega-\lim diam([v_n,w_n]\cap [v(x^i_n),v(y^i_n)])=\infty$. By Remark \ref{nice} and the properties listed above, we then get $p\in\delta_i$ and hence $d(p,\alpha)\leq \epsilon$. As this is true for any $\epsilon$ (and $\alpha$ is closed) we get $p\in\alpha$.
\end{proof}

We can now conclude the proof. The proposition tells us that any piece is obtained as the closure of an union of subspaces isometric to $Z\times \R$. Those subspaces are ultralimits of vertex space of $\tilM$, and therefore each of them corresponds to a vertex $v$ in the ultrapower $\calT$ of the Bass-Serre tree of $M$, that is to say $T^\N/_\sim$ where $(x_n)\sim (y_n)$ if and only if $\omega(\{n:x_n=y_n\})=1$.
Notice that $\calT$ is in a natural way a simplicial forest. Let $V$ be the collection of all vertices $v$ as above. Once again in view of the proposition, $V$ is the set of vertices of some connected subset $T$ of $\calT$, which is therefore a tree.
\par
Finally, property $(\ast)$ follows from the fact that any vertex space $X_v=Y_v\times\R$ of $\tilM$ has the property that $Y_v$ is asymptotically tree-graded with respect to the collection of its boundary components $B$ such that $B\times\R=X_v\cap X_{v'}$ for some vertex space $X_{v'}\neq X_v$.
\end{proof}


\begin{prop}\label{pieces2}
 If $X(T,\{\calF_v\})$ and $X(T',\{\calF'_v\})$ are clusters satisfying condition $(\ast)$, then they are isometric.
\end{prop}

\begin{proof}
 Notice that it is enough to show that $\hat{X}=\hat{X}(T,\{\calF_v\})$ is isometric to $\hat{X'}=\hat{X}(T',\{\calF'_v\})$. The objects associated to $X(T,\{\calF_v\})$ (resp. $X(T,\{\calF_v\})$) described in the definition of cluster will be denoted by $Z_v, Q_v,\gamma_{v,e}$ (resp. $Z'_v, Q'_v,\gamma'_{v,e}$). For each subtree $U$ of $T$, denote $\bigcup_{v\in vert(U)} Q_v/_\sim$ by $X_U$, and similarly for subtrees of $T'$.
\par
A triple $(U,\phi, \psi)$ is said to be \emph{good} if
\begin{enumerate}
 \item $U$ is a subtree of $T$,
 \item $\phi:X_U\to \hat{X'}$ is an isometric embedding,
 \item $\psi:U\to T'$ is a simplicial embedding, and $\phi(Q_v)=Q_{\psi(v)}$,
 \item for each vertex $v$ of $U$, $\phi|_{Q_v}$ preserves the product structure,
 \item for each vertex $v$ of $U$, $\phi|_{Q_v}$ induces a bijection between $\calF_v$ and $\calF'_{\psi(v)}$ (we will denote by $\phi$ also the induced bijections $\calF_v\to\calF'_{\psi(v)}$).
\end{enumerate}

Let us now show the simple fact that if $(T,\phi, \psi)$ is a good triple, then  $\phi$ is an isometry. Indeed, suppose by contradiction that $\phi$ is not surjective. It is then quite clear that $\psi$ is not surjective either. Let $\psi(v)$ be a vertex in $\psi(T)$ such that there exists an edge $e'$ containing $\psi(v)$ but not contained in $\psi(T)$. Condition $5)$ guarantees that there exists an edge $e$ of $T$ containing $v$ such that $\phi(\gamma_{v,e})=\gamma'_{\psi(v),e'}$. Now, we clearly have $\psi(e)=e'$, so that $e'\subseteq \psi(T)$, a contradiction.
\par
So, we need to find a good triple $(T,\phi, \psi)$. We can apply Zorn's Lemma to the natural ordering $\prec$ on good triples (the one given by strict inclusion on the first factor and extension on the other ones), and the only non-trivial thing to show is following claim.
\par
\begin{claim}
 Given a good triple $(U,\phi, \psi)$ where $U$ is a proper subtree of $T$, there exists a good triple $(\overline{U},\overline{\phi},\overline{\psi})\succ (U,\phi, \psi)$.
\end{claim}

Consider a triple $(U,\phi, \psi)$ as above. Let $v$ be a vertex of $T$ which does not lie on $U$, but such that there is an edge $e$ containing $v$ with the other endpoint $w$ on $U$. Set $\overline U=U\cup e$. Also, if $\phi(\gamma_{w,e})=\gamma'_{\psi(w),e'}$, let $v'=\overline{\psi}(v)$ be the endpoint of $e'$ which is not $\psi(w)$ (and set $\overline{\psi}|_{U}= \psi$). Notice that $\overline{\psi}$ still defines an embedding.
\par
We are only left with extending $\phi$ to $\overline{\phi}$. Indeed, this needs to be done only in $Q_v$, where $\phi$ is defined only on $\gamma_{v,e}\times\R$. By $3)$ and the way $Q_v$ is glued to $Q_w$, $\phi$ induces an isometric embedding $\theta:\gamma_{v,e}\to Z'_{\overline{\psi}(v)}$. If we manage to extend $\theta$ to an isometry $\overline{\theta}:Z_v\to Z'_{\overline{\psi}(v)}$, we can then define $\overline{\psi}=\overline{\theta}\times Id_\R$. If we also make sure that $\overline{\theta}$ maps bijectively $\{\gamma_{v,e}\}$ to $\{\gamma'_{v,e}\}$, then all conditions are satisfied. The existence of $\overline{\theta}$ follows from Lemma \ref{theta}.

\end{proof}

\end{document}